\documentclass[10pt]{amsart}

\usepackage{hyperref}
\usepackage{color}
\usepackage{graphicx}
\usepackage{latexsym}
\usepackage{enumerate}
\usepackage{cite}
\usepackage{tikz}
\usetikzlibrary{decorations.pathmorphing}
\usetikzlibrary{decorations.markings}
\usetikzlibrary{shapes,positioning}
\usetikzlibrary{patterns}
\usepackage[text={4.5in,7.2in},centering]{geometry}

\theoremstyle{definition}
\newtheorem{theorem}{Theorem}[section]
\newtheorem{corollary}[theorem]
{Corollary}
\newtheorem{lemma}[theorem]{Lemma}

\newtheorem{proposition}[theorem]{Proposition}

\newtheorem{definition}[theorem]{Definition}

\newtheorem{observation}[theorem]{Observation}

\numberwithin{figure}{section}

\title{Cycle domination, independence, and irredundance in Graphs}

\author{Amy Grady, Fiona Knoll, Renu Laskar, Drew J. Lipman \\ Clemson University}

\begin{document}

\begin{abstract}  A set $S$ of vertices in a graph $G = (V, E)$ is called {\em cycle independent} if the induced subgraph
$\langle S\rangle$ is acyclic, and called {\em odd-cycle indepdendet} if $\langle S\rangle$ is bipartite.
A set $S$ is {\em cycle dominating} (resp. {\em odd-cycle dominating}) if for every vertex $u \in V \setminus S$ there exists a vertex $v \in S$ such that $u$ and $v$ are contained in a (resp. odd cycle) cycle in $\langle S \setminus \{u\}\rangle$.
A set $S$ is {\em cycle irredundant} (resp. odd-cycle irredundant) if for every vertex $v \in S$ there exists a vertex $u \in V \setminus S$ such that $u$ and $v$
are in a (resp. odd cycle) cycle of $\langle S \setminus \{u\}\rangle$, but $u$ is not in a cycle of $\langle S \cup \{u\} \setminus \{v\}\rangle$.
In this paper we present these new concepts, which relate in a natural way to independence, domination and irredundance in graphs.
In particular, we construct analogs to the domination inequality chain for these new concepts.
\end{abstract}

\maketitle
%%%%%%%%%%%%%%%%%%%%%%%%%%      This is to remove the page number from the bottom of the first page %%%%%%%%%%%%%%%%%%%%%%%
\thispagestyle{empty}

\section{Introduction}

In 1978, Cockayne et al. \cite{CockayneEtAl} first defined what has now become a well-known inequality chain of domination related parameters of a graph  as:
\[ir(G)\leq \gamma(G)\leq i(G)\leq \beta(g)\leq \Gamma(G)\leq IR(G),\]
where $ir$ and $IR$ are the lower and upper irredundance numbers respectively, $\gamma$ and $\Gamma$ are the domination and upper domination numbers respectively, and $i$ and $\beta$ are the independent domination number and the independence number respectively.
Since then, other parameters have been added to the chain.
Two of the key concepts of this inequality chain are the concepts of hereditary properties and ancestral properties.

\begin{definition}
We say a property $P$ is {\em hereditary} if, for all sets $S$ that satisfy $P$, every set $S'\subseteq S$ also satisfies $P$.
\end{definition}

\begin{definition}
We say a property $P$ is {\em ancestral} if, for all sets $S$ that satisfy $P$, every set $S'\supseteq S$ also satisfies $P$.
\end{definition}

For example, the property that characterizes a set of vertices being independent is a hereditary property, as every subset of an independent set is independent.
Similarly, the property that characterizes a set of vertices being a dominating set is an ancestral property, as every superset of a dominating set is dominating.

For our purposes, the key behavior of hereditary and ancestral properties is how they affect finding maximal sets and minimal sets.
\begin{definition}
Let $P$ be a property and $S$ a set that holds for $P$,
\begin{itemize}
\item we say $S$ is {\em 1-minimal} if there does not exist a set $S'\subset S$ that holds for $P$ such that $|S\setminus S'|=1$.
\item we say $S$ is {\em 1-maximal} if there does not exist a set $S'\supset S$ that holds for $P$ such that $|S'\setminus S|=1$.
\end{itemize}
\end{definition}

\begin{observation} Let $P$ be a property and $S$ a set that satisfies $P$.
\begin{itemize}
\item If $P$ is hereditary, then $S$ is maximal if and only if $S$ is 1-maximal.
\item If $P$ is ancestral, then $S$ is minimal if and only if $S$ is 1-minimal.
\end{itemize}
\end{observation}

That is, to prove a set is maximal for a hereditary property, it suffices to show that there is no single element that can be added to the set while maintaining the property.
Similarly, to show a set is minimal for an ancestral property, it suffices to show that there is no single element that can be removed from the set while maintaining the property.

In this paper we construct analogs to the domination chain above by considering hereditary and ancestral properties that are analogs to independence and domination, as well as study some bounds for the resulting properties.
We consider the properties, on sets of vertices, that characterize acyclic, and bipartite graphs.

\section{Cycle Independence and Odd-Cycle Independence}

The first property we look at is the acyclic property.
Observe that this is, in fact, a hereditary property as every induced subgraph of a forest is a forest.

\begin{definition}
Let $G=(V,E)$ and $S\subseteq V$. Then the set $S$ is {\em cycle independent} if $\langle S\rangle$, the subgraph induced by $S$, is acyclic.
\end{definition}

This parameter is closely connected to several well studied parameters.
In particular, we consider the relationships between the decycling number $\nabla(G)$ first introduced in \cite{BeinekeVandell}, the size of the maximum induced tree $t(G)$ introduced in \cite{ErdosSaksSos} and cycle independent sets.

\begin{definition}
A set $S\subseteq V(G)$ is a {\em decycling set} if $G\setminus S$ is cycle-free, that is, $G\setminus S$ is a forest.
The minimum order of a decycling set is called the {\em decycling number} of $G$ and is denoted $\nabla(G)$.
\end{definition}

Observe that a set $S\subseteq V$ is cycle independent if and only if $S=V\setminus S'$ for a decycling set $S'$.

\begin{definition}
Let $t(G)$ denote the maximum size of a subset of vertices of a graph $G$ that induces a tree.
\end{definition}

We are interested in defining analogs to $\beta(G)$ the independence number and $i(G)$ the independent domination number in this new paradigm, as well as some bounds associated with them.

\begin{definition}
Let $G=(V,E)$ be a graph; then
\begin{itemize}
\item $\beta_{cy}(G):=\max\{|S|\ :\ S \mbox{ is cycle independent}\}$ is the {\em cycle independence number} and,
\item $i_{cy}(G):=\min\{|S|\ :\ S\mbox{ is maximally cycle independent}\}$ is the {\em lower cycle independence number}.
\end{itemize}
\end{definition}

Let $G=(V,E)$. Observe that as $\beta_{cy}(G)$ is the maximum size of an induced forest of $G$, $|V|=\beta_{cy}(G)+\nabla(G)$.
Similarly, for any graph $G$, $t(G)\leq \beta_{cy}(G)$.
The other parameter, $i_{cy}$, does not seem to be related to $t(G)$.
In particular, Figure \ref{fig:DoubleStar} has $i_{cy}(G)<t(G)=\beta_{cy}(G)$, while other graphs, e.g., $G$ a set of $n$ isolates, have $t(G)<i_{cy}(G)$.

Now we provide some basic bounds on $i_{cy}$ and $\beta_{cy}$.

\begin{figure}
\begin{center}
\begin{tikzpicture}[every loop/.style={}]
\draw (0,0) circle (.3) node (A) {$u_1$}
(4,0) circle (.3) node (B) {$u_2$}
(2,1.5) circle (.3) node (C) {$v_1$}
(2,3.1) circle (.3) node (D) {$v_n$};
\node at (2,2.4) {$\vdots$};
\draw[line width=.5pt] (A) edge (B);
\draw[line width=.5pt] (A) edge (C);
\draw[line width=.5pt] (A) edge (D);
\draw[line width=.5pt] (B) edge (C);
\draw[line width=.5pt] (B) edge (D);
\end{tikzpicture}
\end{center}
\caption{The figure above characterizes a family of graphs produced by $n$ vertices each only adjacent to $u_1$ and $u_2$, and the edge $u_1u_2$.
This family of graphs is an example where the bounds in Proposition \ref{prop:EasyBounds}, Proposition \ref{prop:CyDomBound} and Proposition \ref{prop:CyIRBound} are tight.
In fact, up to an arbitrary subdivision of the edge $u_1u_2$ and vertices not contained in a cycle, this family characterizes when the lower bounds of these propositions are tight. \label{fig:DoubleStar}}
\end{figure}
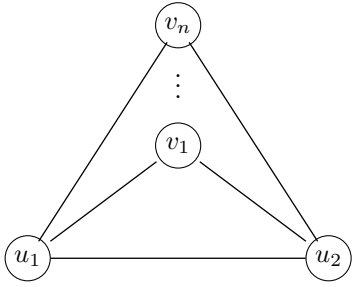

\begin{proposition}\label{prop:EasyBounds}
For a graph $G=(V,E)$, let $\kappa(G)$ denote the number of vertices of $G$ not contained in any cycle of $G$ and $\tau(G)$ be the maximum number of vertex disjoint cycles of $G$. Then
\[\mbox{girth}(G)-1+\kappa(G)\leq i_{cy}(G)\leq \beta_{cy}(G)\leq |V|-\tau(G),\]
where we set $\mbox{girth}(G)=1$ if and only if $G$ is a tree, and to be the girth of $G$ otherwise.
\end{proposition}

\begin{proof}
By definition $i_{cy}(G)\leq \beta_{cy}(G)$ holds for all graphs, it thus suffices to show $\mbox{girth}(G)-1+\kappa(G)\leq i_{cy}(G)$ and $\beta_{cy}(G)\leq |V|-\tau(G)$.

We begin by showing $\mbox{girth}(G)-1+\kappa(G)\leq i_{cy}(G)$.
Suppose $G$ is not a tree, and the smallest cycle of $G$ has size $\mbox{girth}(G)=g$.
Let $S$ be all the vertices of $G$ that are not contained in any cycle of $G$ and $g-1$ additional vertices.
Since the additional $g-1$ vertices cannot form a cycle, $\langle S\rangle$ is acyclic.
Note that this does not depend on the choice of the $g-1$ vertices.
On the other hand, let $v$ be a vertex of $G$ that is not contained in any cycle of $G$.
Every maximal cycle independent set $S$ of $G$ contains $v$.
Moreover, every maximal cycle independent set $S$ contains at least $g-1$ vertices that are contained on some cycle of $G$.
Hence, $|S|\geq \kappa(G)+g-1$.
If $G=(V,E)$ is a tree, then $i_{cy}(G)=|V|$, $\mbox{girth}(G)-1=0$, and $\kappa(G)=|V|$, which gives $i_{cy}(G)= \mbox{girth}(G)-1+\kappa(G)$.

To show $\beta_{cy}(G)\leq |V|-\tau(G)$, it suffices to observe that the decycling number, $\nabla(G)$, is at least $\tau(G)$, since each vertex disjoint cycle has at least one vertex not in the cycle independent set.
\end{proof}

As simple as this bound is, there are cases where the bounds are tight.
For example, Figure \ref{fig:DoubleStar} gives a family of graphs where both bounds are tight.

\begin{corollary}\label{cor:IndepBound}
Let $G$ be a graph where every vertex is contained on a cycle and $\mbox{girth}(G)=g$.
Then, $i_{cy}(G)=g-1$ if and only if $G$ contains a path of length $g-1$ such that every vertex not on the path is adjacent to the leaves and to no other vertices of the path.
\end{corollary}

\begin{proof}
Let $G=(V,E)$ be as above, and $S$ be a cycle independent set of $G$ such that $|S|=i_{cy}(G)=g-1$.
By definition of cycle independence, $\langle S\rangle$ is a forest.
Since the girth of $G$ is $g$, for any vertex $v\notin S$, $\langle S\cup\{v\}\rangle$ contains a cycle of length $g$.
That is, $\langle S\cup\{v\}\rangle\cong C_g$, hence $\langle S\rangle$ is a path.
Since this holds for all vertices $v\notin S$, and $\langle S\rangle$ is independent of the choice of $v$, every vertex $v\notin S$ is adjacent to the leaves of $\langle S\rangle$ and no other vertex in $S$.

Assume $G$ is as above and contains a path $P$ of length $g-1$ such that every vertex not on the path is adjacent to the leaves and no other vertices of the path.
Set $S$ to be the vertices of the path, and observe that $S$ is a maximal cycle independent set since for every $v\notin S$, $\langle S\cup\{v\} \rangle$ contains a cycle.
Thus, by Proposition \ref{prop:EasyBounds}, $|S|=g-1\geq i_{cy}(G)\geq g-1$.
\end{proof}

This corollary gives a characterization of the graphs where $i_{cy}(G)=\mbox{girth}(G) - 1+\kappa(G)$.
In general, graphs with $i_{cy}(G)=\mbox{girth}(G)-1+\kappa(G)$ have, at most, one non-tree component.
The non-tree components are obtained by subdividing the edge $u_1u_2$ of the graph in Figure \ref{fig:DoubleStar}, for some $n$, and then rooting a tree at each vertex.

Moreover, these graphs also satisfy $\beta_{cy}(G)=|V|-\tau(G)=|V|-1$ as the deletion of any vertex on the subdivision of $u_1u_2$ gives a forest.

%%% Odd-cycle %%%
We now consider a more specific case of cycle independence, namely odd-cycle independence.

\begin{definition} A set of vertices $S$ in a graph $G = (V,E)$ is called {\em odd-cycle independent} if the induced subgraph $\langle S \rangle$ contains no odd cycles.
\end{definition}

A set $S$ is odd-cycle independent, i.e., free of odd cycles, if and only if it is bipartite. We again are interested in the analogs of the independence numbers.

\begin{definition} Let $G = (V,E)$ be a graph. Then
	\begin{itemize}
		\item $\beta_{odd} (G) = \max\{|S|: \text{S is odd-cycle independent}\}$ is the {\em odd-cycle independence number}, i.e., the maximum possible number of vertices in an odd-cycle independent set.
		\item $i_{odd} (G) = \min\{|S|: \text{S is maximal odd-cycle independent}\}$ is the {\em lower odd-cycle independence number}, i.e., the minimum possible number of vertices in a maximal odd-cycle independent set.
	\end{itemize}
\end{definition}

For example, in Figure \ref{fig:DoubleStar}, there are three maximal independent sets:\\ $\{u_1,v_1, \ldots, v_n\}$, $\{u_2,v_1, \ldots, v_n\}$ and $\{u_1,u_2\}$. Adding one more vertex to any of these sets would produce an induced graph with an odd cycle, giving us $\beta_{odd}(G) = n+1$ and $i_{odd} = 2$. By definition, we have $i_{odd}(G) \leq \beta_{odd}(G)$, but in this particular case, we see that $i_{odd}(G) < \beta_{odd}(G)$. We now generalize this observation. 

\begin{proposition} \label{prop:IndepIneq} Let $G$ be a connected graphical embedding such that every bounded face is a $C_3$ and the chromatic number $\chi(G) =3$. Then if $|V(G)|=n$ is not a multiple of 3, 
	\[i_{odd}(G) < \beta_{odd}(G).\]
\end{proposition}
\begin{proof}
It suffices to show $i_{odd}(G) \neq \beta_{odd}(G)$. Let $S_1, S_2, S_3$ be set of vertices such that each consists of a color class of $G$, i.e., $S_j \cap S_k = \emptyset$ for $1 \leq j\neq k \leq 3$ and $S_1 \cup S_2 \cup S_3 = V(G)$. Without loss of generality, $|S_1| \geq |S_2| \geq |S_3|$. Let $S_\beta$ be the set of vertices in the two larger color classes, that is $S_\beta = S_1 \cup S_2$. Let $S_i$ be the set of vertices in the two smaller color classes, that is $S_i = S_2 \cup S_3$. 

We claim $S_\beta$ and $S_i$ are maximal odd-cycle independent sets. To prove this claim, consider $v_1 \in  S_3 = V \setminus S_\beta$ and $v_2 \in S_1 = V \setminus S_i$. Since $G$ is a connected graph composed only of triangular faces. $v_1$ and $v_2$ are contained in a $C_3$. Hence, $S_\beta \cup \{v_1\}$ and $S_i \cup \{v_2\}$ is no longer 2-colorable and contains an odd cycle, specifically $C_3$. Thus, $S_\beta$ and $S_i$ are maximal odd-cycle independent sets.

It follows that $i_{odd}(G) \leq |S_i|$ and $\beta_{odd}(G) \geq |S_\beta|$. Since $3$ does not divide $n$, we have our desired result:
	\[i_{odd}(G) \leq |S_i| < |S_\beta| \leq \beta_{odd}(G).\] 
\end{proof}

Now we consider lower and upper bounds for both $i_{odd}(G)$ and $\beta_{odd}(G)$.

\begin{proposition} \label{prop:IndepEasyBound1}
Let $G$ be a graph and $|V(G)|=n$. Let $\kappa_{odd}(G)$ denote the number of vertices of $G$ not contained in any odd cycle of $G$, and let $\tau_{odd}(G)$ be the maximum number of vertex disjoint odd cycles of $G$. Then if G is not bipartite,
	\[\kappa_{odd}(G)+\mbox{girth}(G) -1 \leq i_{odd}(G) \leq n -\tau_{odd}(G);\]
otherwise, $i_{odd}(G) = n$.
\end{proposition}
\begin{proof}
Suppose $G$ is not bipartite, and let $\mbox{girth}(G)=g$. We first show $\kappa_{odd}(G) + g -1 \leq i_{odd}(G)$. Suppose $S$ is a maximal odd-cycle independent set. Then $S$ must contain the set of vertices not contained in any odd cycle of $G$; otherwise, it would not be maximal. Now, suppose $S$ has less than $g-1$ additional vertices. Then, this would contradict the set's maximality since the smallest odd cycle has length of at least $g$. Hence, for any maximal odd-cycle independent set $S$,
	\[|S| \geq \kappa_{odd}(G) +g-1.\]
We now show the inequality $i_{odd}(G) \leq n- \tau_{odd}(G)$. For each disjoint odd cycle, there is at least one vertex not in the odd-cycle independent set. Hence,
	\[i_{odd}(G) \leq n-\tau_{odd}(G).\]

Suppose $G$ is bipartite. Then clearly, $G$ is odd-cycle independent giving us $i_{odd}(G)=n$.

To show the bounds are tight, consider a 3-cycle graph $G$. Then $i_{odd}(G) = 2 = \kappa_{odd}(G) + \mbox{girth}(G) -1 = n - \tau_{odd}(G)$.
\end{proof}

Recall, a graph is bipartite if and only if it is 2-colorable. Hence, we provide the following bound in terms of the chromatic number for the odd-cycle independence number, $\beta_{odd}$.

\begin{proposition}  \label{prop:IndepEasyBound2}
Let $G$ be a graph with $|V(G)|=n$ and $\chi(G)=k$. Then 
	\[ 2 \left\lfloor \frac{n}{k} \right\rfloor \leq \beta_{odd}(G) \leq n - \tau_{odd}(G).\]
\end{proposition}
\begin{proof}
To show the lower bound, consider the color classes $S_1, \ldots, S_k$. Then we have the average number of vertices in each color class is $\frac{n}{k}$. Without loss of generality, $|S_1| \geq |S_2| \geq \dots \geq |S_k|$. Consider the set $S = S_1 \cup S_2$. The set $S$ is odd-cycle independent since it is 2-colorable, i.e., bipartite. Also, $|S| \geq 2 \left\lfloor  \frac{n}{k} \right\rfloor$. It follows that 
	\[\beta_{odd}(G) \geq 2 \left\lfloor  \frac{n}{k} \right\rfloor.\]
To show the upper bound, note that for each disjoint odd cycle, there is at least one vertex not represented in all of the odd-cycle independent sets. Thus,
	\[\beta_{odd}(G)  \leq n-  \tau_{odd}(G).\]
To show the bounds are tight, let $G$ be a 3-cycle. Then $\beta_{odd}(G) = 2 = 2 \left\lfloor\frac{n}{k} \right\rfloor = n - \tau_{odd}(G)$.
\end{proof}

In the case of a maximal outer planar graph $G$, a tighter upper bound can be obtained.

\begin{proposition} \label{prop:IndepSpecificBound}
If $G$ is maximal outer planar, then $\beta_{odd}(G) = |S_1| + |S_2|$ where $S_1$ and $S_2$ are the two largest color classes possible.
\end{proposition}
\begin{proof} Since $G$ is maximal outer planar, $\chi(G)=3$ and every bounded face is a triangle. Let $S_1, S_2, S_3$ be the three color classes. Without loss of generality, $|S_1| \geq |S_2| \geq |S_3|$. 

We first show $\beta_{odd}(G) \geq |S_1|+|S_2|$. Let $S = S_1 \cup S_2$. Since $\langle S\rangle$ is 2-colorable, it is bipartite and hence, odd-cycle independent. Thus, $\beta_{odd}(G) \geq |S_1| + |S_2|$. 

To show $\beta_{odd}(G) \leq |S_1|+|S_2|$, consider an arbitrary odd-cycle independent set $S$. Since every bounded face in $G$ is $C_3$, $S$ may contain no more than two vertices from each face. In other words, $S$ may contain no more than two colors for each face. So, a set $S$ may contain only a subset of two of the three color classes. Hence,
	\[|S| \leq |S_1|+|S_2|.\]
\end{proof}

Finding the parameters of a maximal odd-cycle independent set is closely related to solving the bipartization problem. The bipartization problem consists of minimizing the number of vertices required to be in a set $W$, such that the induced graph $\langle G \setminus W \rangle$ is bipartite. In 1978, the bipartization problem was proven to be NP hard by Yannakakis in \cite{Yannakakis}, but since then there has been success in finding upper and lower bounds.

\begin{definition}
Let $G =(V,E)$. 
	\begin{itemize}
		\item A set $W \subseteq V(G)$ is called an {\em odd-cycle cover} if the induced graph $\langle G \setminus W \rangle$ is bipartite. In addition, we define $\tau = \min\{|W|: G \setminus W \text{ is bipartite}\}$. 
		\item Let $T$ be a collection of vertex disjoint odd cycles. Then $T$ is called a {\em packing} of $G$, and we define $\nu = \max\{|T|: T \text{ is a packing set} \}$.
	\end{itemize}
\end{definition}

In \cite{KralVoss}, it has been shown that for a graph $G$, $ \tau  \geq \nu$, and for a plane graph, $\tau \leq 2 \nu$. 

In terms of the odd-cycle independence number, $\beta_{odd}$, we have $\beta_{odd}(G)=|V(G)|-\tau$. Hence, these proven constraints provide us with additional constraints:
	\begin{itemize}
		\item For a general graph $G$, $\beta_{odd}(G) \leq |V(G)|-\nu$, and
		\item For a plane graph $G$, $\beta_{odd}(G) \geq |V(G)|-2\nu$.
	\end{itemize}

\section{Cycle Domination and Odd-Cycle Domination}
In this section we introduce cycle domination and odd-cycle domination.
The relationship between these and cycle independence and odd-cycle independence is analogous to the relationship between independence and domination.
We also relate these new graph parameters to known existing parameters.

\begin{definition}
Let $G=(V,E)$ and $S\subseteq V$. Then $S$ is a {\em cycle dominating} set if for all $u\notin  S$, there is a cycle $C\subseteq \langle S\cup\{u\}\rangle$ containing $u$.
\end{definition}

Observe that, like domination, cycle domination is an ancestral property.

Double domination was first introduced in \cite{HararyHanes}, and can be used to give a bound on $\gamma_{cy}(G)$.

\begin{definition}
Let $G=(V,E)$ and $S\subseteq V$. Then $S$ is a {\em $k$-tuple dominating} set if for all $u\notin S$, there are $k$ elements of $S$ that are adjacent to $u$.
Denote the size of the smallest $k$-tuple dominating set as $\gamma_k(G)$.
\end{definition}

Observe that every cycle dominating set is in fact a 2-tuple dominating set.
That is, for every vertex not in a cycle dominating set, the vertex is adjacent to at least the two vertices in the set that give the cycle.

\begin{definition}
Let $G=(V,E)$ be a graph then:
\begin{itemize}
\item $\gamma_{cy}(G):=\min\{|S|\ :\ S \mbox{ is cycle dominating}\}$ is the {\em cycle domination number} and,
\item $\Gamma_{cy}(G):=\max\{|S|\ :\ S\mbox{ is minimal cycle dominating}\}$ is the {\em upper cycle domination number}.
\end{itemize}
\end{definition}

This displays a relation between the 2-tuple domination number of a graph $G$ and cycle domination number of $G$.

\begin{proposition}
Let $G$ be a graph then, $\gamma_{2}(G)\leq \gamma_{cy}(G)$.
\end{proposition}
\begin{proof}
It suffices to observe that any cycle dominating set contains a minimal 2-tuple dominating set.
This follows from the fact that any cycle dominating set is a 2-tuple dominating set, and hence contains a minimal 2-tuple dominating set.
\end{proof}

Another lower bound can be obtained for $\gamma_{cy}(G)$.

\begin{proposition}\label{prop:CyDomBound}
Let $G$ be a graph, and let $\kappa(G)$ be the number of vertices that are not contained in any cycles of $G$. Then
\[\mbox{girth}(G)-1+\kappa(G)\leq \gamma_{cy}(G),\]
where $\mbox{girth}(G)=1$ if and only if $G$ is a forest.
\end{proposition}
\begin{proof}
Observe that every vertex that is not contained in a cycle of $G$ is in every maximal dominating set.
Thus, we may assume that every vertex is contained in such a cycle, that is, $\kappa(G)=0$ and $\mbox{girth}(G)=g\geq 3$.
Now, suppose $S$ is a minimal cycle dominating set.
If $\langle S\rangle$ has a cycle, then $|S|\geq g$. Else if $\langle S\rangle$ is a forest, then for every vertex $v\notin S$, $\langle S\cup\{v\}\rangle$ has a cycle. Hence $|S\cup\{v\}|\geq g$ and thus, $|S| \geq g-1$.
\end{proof}

As before, we can characterize when this lower bound is tight.

\begin{corollary}
Let $G$ be a graph so that every vertex is contained on a cycle, and let $\mbox{girth}(G)=g$.
Then, $\gamma_{cy}(G)=g-1$ if and only if $G$ contains a path of length $g-1$ such that every vertex not on the path is adjacent to the two leaves and no other vertex of the path.
\end{corollary}
\begin{proof}
Assume $G$ is a graph such that every vertex is contained on a cycle.
Let $S$ be a cycle dominating set such that $|S|=g-1$.
By the definition of girth, $S$ is a cycle independent set.
By Corollary \ref{cor:IndepBound}, we conclude that $G$ contains an induced path such that every vertex not on the path is adjacent to the leaves and no other vertex.

Now, assume $G$ has a path of length $g-1$ such that every vertex not on the path is adjacent to the two leaves and no other vertex of the path.
Let $S$ be the vertex set of this path.
Observe that for every $u\notin S$, $\langle S\cup\{u\}\rangle$ is a cycle containing $u$.
Thus, $S$ is a cycle dominating set, further, $S$ is minimal as the deletion of any vertex of $S$ yields a set that is not cycle dominating.
By the bound in Proposition \ref{prop:CyDomBound}, we know that $|S|=\gamma_{cy}(G)$.
\end{proof}

As with cycle independence, this gives us a characterization for when the bound in Proposition \ref{prop:CyDomBound} is tight.
In fact, the condition for the lower bound in Proposition \ref{prop:EasyBounds} to be tight is the same as the condition for the lower bound in Proposition \ref{prop:CyDomBound}.

Finally, we build the next step of the domination inequality chain for cycle domination.

\begin{lemma}\label{lem:CyIneqChain2}
For any graph $G$,
\[\gamma_{cy}(G)\leq i_{cy}(G)\leq \beta_{cy}(G)\leq \Gamma_{cy}(G).\]
\end{lemma}
\begin{proof}
It suffices to observe that any maximal cycle independent set is in fact a minimal cycle dominating set.
This is due to the fact that given a maximal cycle independent set $S$ of $G$, and a vertex $v\notin S$, there is a cycle containing $v$ in $\langle S\cup\{v\}\rangle$.
\end{proof}

%%% Odd-Cycle Dominating %%%
Now we consider similar questions using odd-cycle domination instead of cycle domination.

\begin{definition}A set $S \subseteq V(G)$ is called {\em odd-cycle dominating} if for every vertex $v \in V \setminus S$ there exists $u \in S$, such that $\{v\} \cup \{u\} $ is contained in an odd cycle of the induced subgraph $\langle S \cup \{v\} \rangle$. \end{definition}

The odd-cycle domination number is defined as follows:

\begin{definition} Let $G = (V,E)$ be a graph. Then
	\begin{itemize}
		\item $\Gamma_{odd} (G) = \max\{|S|: \text{S is minimal odd-cycle dominating}\} $ is the {\em odd-cycle domination number}, i.e., the maximum possible number of vertices in a minimal odd-cycle dominating set.
		\item $\gamma_{odd} (G) = \min\{|S|: \text{S is minimal odd-cycle dominating}\}$ is the {\em lower odd-cycle domination number}, i.e., the minimum possible number of vertices in an odd-cycle dominating set.
	\end{itemize}
\end{definition}

For example, consider Figure \ref{fig:DoubleStar}. The sets $\{u_1,v_1, \ldots, v_n\}$, $\{u_2,v_1, \ldots, v_n\}$ and $\{u_1,u_2\}$ are minimal odd-cycle dominating. In addition, any set containing $\{u_1,u_2\}$ is an odd-cycle dominating set.

The minimal sets mentioned in the previous example were initially presented as maximal odd-cycle independent sets, but are now seen to be minimal dominating sets. This observation holds for every graph $G$ as shown in the ensuing lemma.

\begin{lemma} \label{Lemma:IndepDom}
For a graph $G=(V,E)$, if $S \subset V(G)$ is a maximal odd-cycle independent set, then $S$ is a minimal odd-cycle dominating set.
\end{lemma}
\begin{proof}
Suppose $S$ is a maximal odd-cycle independent set. Then for any $v\in V \setminus S$, the induced graph $\langle S \cup \{v\} \rangle$ contains an odd cycle. Thus, we have $S$ is an odd-cycle dominating set. But also note that since $S$ is a maximal odd-cycle independent set, for any $s\in S$, the induced graph $\langle S \setminus s \rangle$ does not dominate $s$. Thus, $S$ is a minimal odd-cycle dominating set.
\end{proof}

Note that the converse of Lemma \ref{Lemma:IndepDom} is not necessarily true, as can be seen by the following example: consider the graph G in Figure \ref{fig:SunGraph}; the set of vertices $S=\{2,4,6\}$ is a minimal odd-cycle dominating set, but is not independent.  

		\begin{figure}
		\begin{center}
		\begin{tikzpicture} [every loop/.style={}] 
		\node [draw,circle] (A) at (4,0) {1};
		\node [draw,circle] (B) at (2,0) {2};
		\node [draw,circle] (C) at (0,0) {3};
		\node [draw,circle] (D) at (1,1) {4};
		\node [draw,circle] (E) at (2,2) {5};
		\node [draw,circle] (F) at (3,1) {6};
		\draw[line width=1pt] (C) edge (D);
		\draw[line width=1pt] (D) edge (E);
		\draw[line width=1pt] (E) edge (F);
		\draw[line width=1pt] (A) edge (F);
		\draw[line width=1pt] (A) edge (B);
		\draw[line width=1pt] (B) edge (C);
		\draw[line width=1pt] (B) edge (D);
		\draw[line width=1pt] (B) edge (F);
		\draw[line width=1pt] (D) edge (F);
		\end{tikzpicture}
		\end{center}
		\caption{This figure gives an explicit example of minimal dominating sets, which are not maximal independent and an example of when $\gamma_{odd}(G) < i_{odd}(G)$. \label{fig:SunGraph}}
		\end{figure}
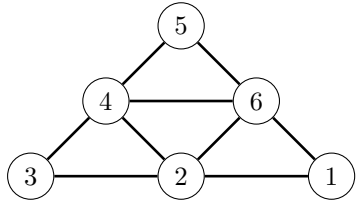

The preceding example gives an instance  when $\gamma_{odd}(G) <i_{odd}(G)$. This observation is generalized in the next proposition.

\begin{proposition} \label{prop:GammaIIneq}
For a graph $G$,
	\[\beta_{odd}(G) \leq \Gamma_{odd}(G) \quad \text{and}  \quad \gamma_{odd} (G) \leq  i_{odd} (G)\]
\end{proposition}
\begin{proof}
It follows from Lemma  \ref{Lemma:IndepDom}.
\end{proof}

Now that we have a better understanding of odd-cycle dominating sets and how they relate to odd-cycle independent sets, we prove some simple bounds, which are shown to be tight.

\begin{proposition} \label{prop:GammaEasyIneq}
Let $G=(V,E)$ be a graph and $|V(G)|=n$. Let $\tau_{odd}(G)$ be the maximum number of vertex disjoint odd cycles of $G$. Then
	\[ 2 \left\lfloor\frac{n}{k} \right\rfloor \leq \Gamma_{odd}(G) \leq n - \tau_{odd}(G).\] 
\end{proposition}
\begin{proof} 
To prove the lower bound, recall 
	\[\Gamma_{odd}(G) \geq \beta_{odd}(G) \geq 2 \left\lfloor \frac{n}{k} \right\rfloor.\]
For the upper bound, consider a minimal dominating set $S$. Suppose $|S| > n-\tau_{odd}(G)$. Then there is a disjoint odd cycle $C$ contained in $S$. Consider $\langle S \setminus \{v\} \rangle$ where $v \in C$. Note that the vertex is still dominated by the set $S \setminus \{v\}$, i.e., $v$ is contained in an odd cycle of the induced subgraph $S$. Also note that since $C$ is a vertex disjoint odd cycle, then $\{v\}$ is not essential for dominating another element. This contradicts the minimality of $S$. Thus, $\Gamma_{odd}(G) \leq n- \tau_{odd}(G)$. 

To show the bounds are tight, let $G$ be a 3-cycle. Then $\Gamma_{odd}(G) = 2 = 2 \left\lfloor\frac{n}{k} \right\rfloor = n - \tau_{odd}(G)$.
\end{proof}

Similarly, we obtain a bound for $\gamma_{odd}$.

\begin{proposition} \label{prop:LittleGammaEasyIneq}
Let $G$ be a graph and $|V(G)|=n$. Let $\kappa_{odd}(G)$ denote the number of vertices of $G$ not contained in any odd cycle of $G$, and let $\tau_{odd}(G)$ be the maximum number of vertex disjoint odd cycles of $G$. Then if G is not bipartite,
	\[ \mbox{girth}(G) -1 + \kappa_{odd}(G) \leq \gamma_{odd}(G) \leq n - \tau_{odd}(G);\] 
otherwise, $\gamma_{odd}(G) = n$.
\end{proposition}
\begin{proof} 
The upper bound is trivial as $\gamma_{odd}(G) \leq \Gamma_{odd}(G) \leq n- \tau_{odd}(G)$. For the lower bound, consider an odd-cycle dominating set $S$, and let $\mbox{girth}(G)=g$. Since $S$ is an odd-cycle dominating set, all of the vertices not contained in an odd cycle of $G$ must be in $S$. We claim an additional $g-1$ vertices must be in $S$. Suppose not. Consider the smallest cycle $C$ and note that $C$ must have length of at least $g$. Since $S$ has $g-2$ or less additional vertices contained in an odd cycle, we can assume at least two of the vertices in $C$ are not in $S$. Suppose these vertices are dominated by $S$. Then each must be contained in a separate odd cycle where all other vertices are in $S$, but this contradicts the fact that $\mbox{girth}(G) =g$. Thus, $S$ has at least $g-1$ additional vertices giving us
	\[\gamma_{odd}(G) \geq \kappa_{odd}(G) + \mbox{girth}(G) -1.\]
To show that these bounds are tight, we consider the same example as we did in Proposition \ref{prop:GammaEasyIneq}. Consider a 3-cycle graph $G$. Then \[\gamma_{odd}(G) = 2 = \mbox{girth} -1 +\kappa_{odd}(G) = n - \tau_{odd}(G).\]

\end{proof}

\section{Cycle Irredundance and Odd-cycle Irredundance}
In this section we define the analogs of irredundance for both cycle domination and odd-cycle domination.

\begin{definition}
Let $G=(V,E)$ be a graph and $S\subseteq V$. Then $S$ is {\em cycle irredundant} if for all $u\in S$, there is a $v\in S^c\cup\{u\}$ such that
\begin{itemize}
\item $v=u$, if $u$ is not contained in a cycle of $\langle S\rangle$,
\item or, $v$ is not contained in a cycle of $\langle \left(S\setminus\{u\}\right)\cup\{v\}\rangle$, but is contained in a cycle of $\langle S\cup\{v\}\rangle$.
\end{itemize}
\end{definition}

That is, for every $u\in S$, $u$ is needed for $S$ to cycle dominate a vertex $v$.
Observe that if $u$ is not contained in a cycle of $\langle S\rangle$, then $v=u$ is an example of such a vertex.

Unlike independence and domination, cycle irredundance is not a hereditary or ancestral property.
Thus, checking for maximality with respect to cycle irredundance cannot be done by checking if there is a single vertex that can be added to the set.

Now, as before, we define a pair of parameters for cycle irredundance.

\begin{definition}
Let $G=(V,E)$ be a graph then:
\begin{itemize}
\item $IR_{cy}(G):=\max\{|S|\ :\ S \mbox{ is cycle irredundant}\}$ is the {\em upper cycle irredundance number} and,
\item $ir_{cy}(G):=\min\{|S|\ :\ S\mbox{ is maximal cycle irredundant}\}$ is the {\em lower cycle irredundance number}.
\end{itemize}
\end{definition}

As was the case for cycle independence and cycle domination, we are interested in exploring bounds on these parameters.
In particular, we are interested in finding a lower bound on $ir_{cy}(G)$.

\begin{proposition}\label{prop:CyIRBound}
Let $G$ be a graph, and let $\kappa(G)$ be the number of vertices that are not contained in any cycles of $G$. Then
\[\mbox{girth}(G)-1+\kappa(G)\leq ir_{cy}(G),\]
where $\mbox{girth}(G)=1$ if and only if $G$ is a forest.
\end{proposition}
\begin{proof}
Observe that if a vertex $v$ is not contained in a cycle of $G$ then for any cycle irredundant set $S$, $S\cup\{v\}$ is still cycle irredundant, as $v$ cycle dominates itself, and is not in any cycles of $\langle S\cup\{v\}\rangle$.

Thus, we may assume that every vertex of $G$ is contained on a cycle, and hence $\kappa(G)=0$.
Let $\mbox{girth}(G)=g\geq 3$.
Observe that for any set $S$, such that $|S|\leq g-1$, $\langle S\rangle$ is a forest.
Hence, by the definition of cycle irredundance, $S$ is cycle irredundant.
Thus, $|S|=g-1\leq ir_{cy}(G)$.
\end{proof}

\begin{corollary}
Let $G$ be a graph such that every vertex is contained on a cycle, and let $\mbox{girth}(G)=g$.
Then, $ir_{cy}(G)=g-1$ if and only if $G$ contains a path of length $g-1$ such that every vertex not on the path is adjacent to the two leaves and no other vertex of the path.
\end{corollary}
\begin{proof}
Suppose $S$ is a maximal cycle irredundant set such that $|S|=g-1$.
Since $S$ is maximal cycle irredundant, every vertex $v\notin S$ is adjacent to at least one vertex in $S$.
Since $\langle S\rangle$ is a forest, for every $v\notin S$, adding $v$ to $S$ creates a set that is not cycle irredudant.
There are two cases, either adding $v$ to $S$ creates a cycle, or it does not.
If adding $v$ to $S$ does not create a cycle, then there is a vertex $u\notin S$, $u\neq v$, such that $\langle S\cup\{u,v\}\rangle$ has two cycles containing $u$: one containing $v$ and one not containing $v$.
As $\langle S\cup\{v\}\rangle$ is a forest, $u$ is adjacent to vertices in at least two components of $\langle S\rangle$ that are not separate components in $\langle S\cup\{v\}\rangle$.
Hence, the cycle that contains $u$ in $\langle S\cup\{u\}\rangle$ cannot contain all the vertices of $S$.
That is, the cycle has length less than $g$, which contradicts the definition of $g=\mbox{girth}(G)$.

Now suppose $\langle S\cup\{v\}\rangle$ is not a forest for all vertices $v\notin S$. Then by definition of $g$, $\langle S\cup\{v\}\rangle$ contains a cycle.
That is, there is a path of length $g-1$ so that every vertex not on the path is adjacent to exactly the two leaves of the path.

%Now suppose there is a path of length $g-1$ so that every vertex not on the path is adjacent to exactly the two leaves of the path.
Consider such a path.
Clearly, the vertices of the path give a maximal cycle irredundant set, as they give a cycle dominating set.
It suffices to show that there are no smaller sets that are still maximal cycle irredundant.
Suppose $S'$ is an irredundant set with $|S'|<g-1$, we show that $S'$ is not maximal.
Clearly, if there is a vertex $v$ that is not adjacent to at least two vertices of $S'$, then $S'\cup\{v\}$ is still irredundant.
Suppose $v$ is not a vertex of $S'$ and $\langle S'\cup\{v\}\rangle$ is a forest, i.e., $v$ connects at least two components of $\langle S'\rangle$.
Then, there is a vertex $u\notin S'\cup\{v\}$ that has at least two cycles in $\langle S'\cup\{u,v\}\rangle$ containing $u$: one containing $v$ and one not containing $v$.
The cycle not containing $v$ has length less than $g-1$, which contradicts the definition of $g$.
Hence, for all $v$ not in $S'$, $\langle S'\cup\{v\}\rangle$ contains a cycle which contradicts the definition of $g$.
Thus, $S$ is the smallest maximal cycle irredundant set.
\end{proof}

As before, this corollary gives a characterization of when $ir_{cy}(G)=g-1$; namely, the same as the characterization of when $\gamma_{cy}(G)=g-1$.

The full analog of the domination inequality chain for cycle domination can now be stated.

\begin{theorem}
Let $G$ be a graph. Then
\[ir_{cy}(G)\leq \gamma_{cy}(G)\leq i_{cy}(G)\leq \beta_{cy}(G)\leq \Gamma_{cy}(G)\leq IR_{cy}(G).\]
\end{theorem}
\begin{proof}
By Lemma \ref{lem:CyIneqChain2} we know that $\gamma_{cy}(G)\leq i_{cy}(G)\leq \beta_{cy}(G)\leq \Gamma_{cy}(G)$.
Thus, it suffices to show that $ir_{cy}(G)\leq \gamma_{cy}(G)$ and $\Gamma_{cy}(G)\leq IR_{cy}(G)$.
To prove this, we show that any minimal cycle dominating set is in fact maximal cycle irredundant.

Suppose $S$ is a minimal cycle dominating set, and let $v\in S$; we show that $S$ is cycle irredundant.
There are two cases: either $v$ is contained in a cycle of $\langle S\rangle$, or it is not contained in such a cycle.
If $v$ is not contained in a cycle of $\langle S\rangle$, then the vertex $v$ is needed to cycle dominate itself.
If $v$ is contained in such a cycle, then by the definition of minimality, there is a vertex $u\notin S$ such that every cycle of $\langle S\cup\{u\}\rangle$ containing $u$ also contains $v$.
Hence $S$ is a cycle irredundant set.

To show maximality, suppose there is a set $S'$ such that $S\subseteq S'$ and $S'$ is maximal irredundant.
Since $S'$ contains $S$, there is a vertex $u\notin S'$ such that $u$ is contained in a cycle of $\langle S'\cup\{u\}\rangle$ but not in $\langle S\cup\{u\}\rangle$. This contradicts the definition of cycle domination.
Hence $S$ is maximal cycle irredundant.
\end{proof}

%%%% Odd-Cycle Irredundance %%%
We now look at the relaxation of odd-cycle domination, odd-cycle irredundance.

\begin{definition} A set $S$ is {\em odd-cycle irredundant} if for every $v \in S$ there exists $u \in V \setminus (S-v)$ such that $S$ dominates $u$, but $S \setminus \{v\}$ does not dominate $u$. \end{definition}

Given the nature of odd-cycle irredundant sets, we are interested in the maximal odd-cycle irredundant sets. Consequently, we define the upper and lower odd-cycle irredundance numbers.

 \begin{definition} Let $G = (V,E)$ be a graph. Then
 \begin{itemize}
 	\item  $ IR_{odd} (G) = \max\{|S|: \text{S is odd-cycle irredundant}\} $ is the {\em upper odd-cycle irredundance number} and,
 	\item $ir_{odd} (G) = \min\{|S|: \text{S is maximal odd-cycle irredundant}\}$ is the {\em lower odd-cycle irredundance number}.
\end{itemize}
 \end{definition}
 
As was the case for the lower cycle irredundant number, we are able to obtain a tight bound for the lower odd-cycle irredundant number.

\begin{proposition}
Let $G$ be a graph, and let $\kappa_{odd}(G)$ be the number of vertices that are not contained in any odd cycles of $G$. Then, if $G$ is bipartite,
\[\mbox{girth}(G)-1+\kappa_{odd}(G)\leq ir_{odd}(G);\]
otherwise, $ir_{odd}(G) = n$.
\end{proposition}
\begin{proof}
Observe that for any vertex $v$ not contained in an odd cycle of $G$, then for any odd-cycle irredundant set $S$, $S \cup \{v\}$ is still odd-cycle irredundant. Hence, $|S| \geq \kappa_{odd}(G)$. Let $\mbox{girth}(G) = g$.
We want to show for a maximal odd-cycle irredundant set $S$, there is an additional $g-1$ vertices that are contained in odd cycles of $G$. Suppose not. Then $S$ has only $g-2$ or less vertices contained in odd cycles of $G$. Note that each of these vertices are only dominated by itself as the $\mbox{girth}(G) = g$. Consider a vertex $v \notin S$ contained in one of these odd cycles. Note that $S$ does not dominate $v$ and that $\langle	S \cup \{v\} \rangle$ does not contain an odd cycle, as $\mbox{girth} = g$ and there are only $g-1$ vertices in $\langle S \cup \{v\} \rangle$. Hence, the only element in $\langle S \cup \{v\} \rangle$ to dominate $\{v\}$ is itself. As a result, $S \cup\{v\}$ is odd-cycle irredundant, which contradicts the maximality of $S$. Thus,
	\[|S| \geq \kappa_{odd} + g - 1,\]
giving us our desired result. 

To show the bound is tight, consider a $C_3$ graph $G$. Then $ir_{odd}(G) = 3 = g-1 + \kappa_{odd}(G)$.
\end{proof}

In the case of $C_3$ graphs, the minimal dominating sets are maximal irredundant sets, and vice versa. This observation can be generalized as follows: 

\begin{lemma}  \label{lemma:DomIrr}
For a graph $G=(V,E)$, if $S \subset V(G)$ is a minimal odd-cycle dominating set, then $S$ is a maximal irredundant set.
\end{lemma}
\begin{proof}
Consider a minimal odd-cycle dominating set $S$. Then for any $s \in S$, the induced graph $\langle S \setminus s \rangle$ dominates at least one less vertex than the induced subgraph $\langle S \rangle$. Thus, $S$ is odd-cycle irredundant. Since $S$ is minimal odd-cycle dominating, we also have that for any $v \in V \setminus S$, the vertex $v$ is dominated by $\langle S \rangle$. As a result, $S \cup \{v\}$ is not odd-cycle irredundant, implying $S$ is a maximal odd-cycle irredundant set.  
\end{proof}

Using the containment given in Lemma \ref{lemma:DomIrr}, we obtain the next proposition.

\begin{proposition} \label{prop:gamma_ir_odd}
For a graph $G$,
	\[\Gamma_{odd} (G) \leq IR_{odd} (G) \quad  \text{and}  \quad ir_{odd} (G) \leq \gamma_{odd} (G).\]
\end{proposition}
\begin{proof}
Follows directly from Lemma \ref{lemma:DomIrr}.
\end{proof}

Thus, from the preceding propositions, we obtain for odd cycles an analog of the domination chain.

\begin{theorem}
Let $G$ be a graph. Then
	\[ir_{odd}(G) \leq \gamma_{odd}(G) \leq i_{odd}(G) \leq \beta_{odd}(G) \leq \Gamma_{odd}(G) \leq	IR_{odd}(G).\]
\end{theorem}
\begin{proof}
Follows directly from the definition of $\beta_{odd}$ and $i_{odd}$  and the propositions \ref{prop:GammaIIneq} and \ref{prop:gamma_ir_odd}.
\end{proof}

\section{Acknowledgments}

The authors would like to thank Steve Hedetniemi for his help with this problem.

\bibliography{CycleDominationBib}{}
\bibliographystyle{plain}

\end{document}